\documentclass[a4paper,reqno,10pt]{amsart}

\usepackage{amssymb}
\usepackage{amstext}
\usepackage{amsmath}
\usepackage{amscd}
\usepackage{amsthm}
\usepackage{amsfonts}
\usepackage{enumerate}
\usepackage{graphicx}
\usepackage{latexsym}
\usepackage{mathrsfs}
\usepackage[all,color]{xy}
\usepackage{mathtools}

\usepackage{tikz}
\usepackage{tikz-cd}

\usepackage{hyperref}

\newtheorem{theorem}{Theorem}[section]

\newtheorem{corollary}[theorem]{Corollary}
\newtheorem{lemma}[theorem]{Lemma}
\newtheorem{proposition}[theorem]{Proposition}
\newtheorem{definition-proposition}[theorem]{Definition-Proposition}

\newtheorem{question}[theorem]{Question}

\theoremstyle{definition}
\newtheorem{definition}[theorem]{Definition}

\newtheorem{example}[theorem]{Example}

\theoremstyle{remark}

\newcommand{\mm}{{\mathfrak{m}}}

\newcommand{\pp}{{\mathfrak{p}}}

\newcommand{\depth}{\operatorname{depth}\nolimits}
\newcommand{\rank}{\operatorname{rank}\nolimits}

\newcommand{\Ext}{\operatorname{Ext}\nolimits}

\newcommand{\Hom}{\operatorname{Hom}\nolimits}

\newcommand{\End}{\operatorname{End}\nolimits}

\newcommand{\RHom}{\mathbf{R}\strut\kern-.2em\operatorname{Hom}\nolimits}
\newcommand{\RshHom}{\mathbf{R}\strut\kern-.2em\mathscr{H}\strut\kern-.3em\operatorname{om}\nolimits}
\newcommand{\shHom}{\mathscr{H}\strut\kern-.3em\operatorname{om}\nolimits}

\newcommand{\Spec}{\operatorname{Spec}\nolimits}

\newcommand{\GL}{\operatorname{GL}\nolimits}
\newcommand{\SL}{\operatorname{SL}\nolimits}

\newcommand{\Cl}{\operatorname{Cl}\nolimits}

\DeclareMathOperator{\moduleCategory}{\mathsf{mod}} \renewcommand{\mod}{\moduleCategory}
\DeclareMathOperator{\Mod}{\mathsf{Mod}}
\DeclareMathOperator{\proj}{\mathsf{proj}}

\DeclareMathOperator{\CM}{\mathsf{CM}}
\DeclareMathOperator{\add}{\mathsf{add}}

\DeclareMathOperator{\refl}{\mathsf{ref}}

\newcommand{\cut}{\ar@{-}@[|(5)]}

\numberwithin{equation}{section}

\newcommand{\sfS}{\mathsf{S}}

\newcommand{\QQ}{\mathbb{Q}}

\begin{document}

\title{On steady non-commutative crepant resolutions}  
\author[O. Iyama \and Y. Nakajima]{Osamu Iyama \and Yusuke Nakajima} 

\address[O. Iyama]{Graduate School of Mathematics, Nagoya University, Furocho, Chikusaku, Nagoya 464-8602, Japan}
\email{iyama@math.nagoya-u.ac.jp}
\urladdr{http://www.math.nagoya-u.ac.jp/~iyama/}

\address[Y. Nakajima]{Graduate School Of Mathematics, Nagoya University, Furocho, Chikusaku, Nagoya, 464-8602 Japan 
\footnote{Present address: Kavli Institute for the Physics and Mathematics of the Universe (WPI), UTIAS, The University of Tokyo, Kashiwa, Chiba 277-8583, Japan}} 
\email{m06022z@math.nagoya-u.ac.jp \footnote{Present email: yusuke.nakajima@ipmu.jp}}


\subjclass[2010]{Primary 16G50, 14A22 ; Secondary 13C14,  14E15} 
\keywords{Non-commutative crepant resolutions, class groups, quotient singularities} 

\maketitle

\begin{abstract} 
We introduce special classes of non-commutative crepant resolutions (= NCCR) which we call steady and splitting.
We show that a singularity has a steady splitting NCCR if and only if it is a quotient singularity by a finite abelian group.
We apply our results to toric singularities and dimer models. 
\end{abstract}



\section{Introduction} 

Geometric tilting theory gives a framework to study the derived categories of schemes and stacks in terms of non-commutative algebras obtained as the endomorphism algebras of tilting complexes. A well-studied class of geometric tilting is called McKay correspondence, which are given by resolution of singularities \cite{KV, BKR}. Recently, Van den Bergh introduced the notion of non-commutative crepant resolutions (= NCCR) \cite{VdB1,VdB2} to explain Bridgeland's theorem \cite{Bri} on derived equivalence of crepant resolutions \cite{BO}, from a viewpoint of geometric tilting theory. 
For example, in the case of McKay correspondence, skew group algebras are NCCRs of quotient singularities.

For nice classes of singularities, NCCRs provide a method to construct non-commutative rings derived equivalent to their resolutions directly from the given singularities.
The notion of NCCRs is quite useful thanks to the fact that it naturally appears in Cohen-Macaulay representation theory
initiated by Auslander-Reiten in 70s \cite{Aus2}. In fact, the recent notion of cluster tilting subcategories gives a categorical framework to study NCCRs (see \cite{Iya,IR,IW2,IW4}).

An interesting family of NCCRs is given by a dimer model, which is a quiver with potential drawn on a torus 
and gives us an NCCR of a Gorenstein toric singularity in dimension three (see e.g. \cite{Bei, Boc2, Bro, Dav, IU, MR}). 
For more results and examples of NCCRs, we refer to \cite{BLVdB, BIKR, Dao, DFI, DH, IW1, IW3, TU, Wem}. See also a survey article \cite{Leu}, and the references therein.

\medskip

In this paper, we introduce nice classes of NCCRs which we call steady and splitting, 
and discuss existence of such NCCRs. 
We start with recalling the definition of a non-commutative crepant resolution due to Van den Bergh \cite{VdB2} (see also \cite{VdB1}). 
For further details on terminologies, see later sections. 

\begin{definition}
Let $R$ be a CM ring and $\Lambda$ be a module-finite $R$-algebra. We say 
 \begin{enumerate}
\setlength{\parskip}{0pt} 
\setlength{\itemsep}{0pt}
 \item $\Lambda$ is an \emph{$R$-order} if $\Lambda$ is a CM $R$-module, 
 \item an $R$-order $\Lambda$ is \emph{non-singular} if ${\rm gl.dim}\,\Lambda_\pp={\rm dim}\,R_\pp$ for all $\pp\in\Spec R$. 
 \end{enumerate}
\end{definition}

We refer to \cite[2.17]{IW2} for several conditions which are equivalent to $\Lambda$ is a non-singular $R$-order. 
Using this notion, Van~den~Bergh \cite{VdB2} introduced the notion of NCCR. (Note that unlike \cite{VdB2}, we do not assume that $R$ is Gorenstein in this paper.) 
Also we recall the notion of NCR (see \cite{DITV}). 

\begin{definition}
\label{NCCR_def}
Let $R$ be a CM ring, and $0\neq M\in\refl R$. Let $E:=\End_R(M)$. 
\begin{enumerate}
\setlength{\parskip}{0pt} 
\setlength{\itemsep}{0pt}
\item We say $E$ is a \emph{non-commutative crepant resolution} (= \emph{NCCR}) of $R$ 
or $M$ \emph{gives an NCCR} of $R$ if $E$ is a non-singular $R$-order. 
\item We say $E$ is a \emph{non-commutative resolution} (= \emph{NCR}) of $R$ 
or $M$ \emph{gives an NCR} of $R$ if ${\rm gl.dim}\,E<\infty$. 
\end{enumerate}
\end{definition}

\subsection{Our results}

The existence of an NCCR of $R$ shows $R$ has a mild singularity. 
For example, it was shown in \cite{SVdB} (see also \cite{DITV}) that under mild assumptions, any Gorenstein ring which 
has an NCCR has only rational singularities. 

In this paper, we impose extra assumptions on NCCRs, 
and we will show the existence of such an NCCR characterizes some singularities. 
Our point is that the size of $\End_R(M)$ as an $R$-module becomes much bigger than that of $M$, 
that is, $\rank_R\End_R(M)=(\rank_RM)^2$. 
Therefore, as an $R$-module, $\End_R(M)$ usually has a direct summand which does not appear in $M$. 
Therefore the following class of NCCRs is of interest.  

\begin{definition}
\label{def_steady}
We say a reflexive $R$-module $M$ is \emph{steady} if $M$ is a \emph{generator} (that is, $R\in\add_RM$) and $\End_R(M)\in\add_RM$ holds.
We say an NCCR (resp. NCR) $\End_R(M)$ is \emph{steady NCCR} (resp. \emph{steady NCR}) if $M$ is steady.
\end{definition}

We refer to Lemma~\ref{basic_pro} for basic properties of steady modules.
Also note that the first condition ``$M$ is a generator'' is a consequence of the
second condition ``$\End_R(M)\in\add_RM$'' in many cases (see Lemma~\ref{prop_steady}). 

For example, quotient singularities have a steady NCCR (see Example~\ref{steady_example}). 
A natural question is the converse, that is, singularities having steady NCCRs are quotient singularities.
The aim of this paper is to give a partial answer to this question. 
To consider such a question, we introduce another class of nice NCCRs called splitting, and 
show that existence of steady splitting NCCRs implies quotient singularities. 

\begin{definition}
\label{def_splitting}
We say a reflexive $R$-module $M$ is \emph{splitting} if $M$ is a direct sum of reflexive modules of rank one. 
We say an NCCR (resp. NCR) $\End_R(M)$ is \emph{splitting NCCR} (resp. \emph{splitting NCR}) if $M$ is splitting. 
\end{definition}

There are several examples of splitting NCCRs (see Example~\ref{splitting_example}).

We are now ready to state the main theorem, which will be shown in Section~\ref{sec_main}. 

\begin{theorem}
\label{main_intro} (see Theorem~\ref{main_thm})
Let $R$ be a $d$-dimensional complete local Cohen-Macaulay normal domain containing an algebraically closed field of characteristic zero. 
Then the following conditions are equivalent.
\begin{enumerate}[$\bullet$]
\setlength{\parskip}{0pt} 
\setlength{\itemsep}{0pt}
\item $R$ is a quotient singularity associated with a finite abelian group $G\subset\GL(d, k)$ (i.e. $R=S^G$ where $S=k[[x_1, \cdots, x_d]]$).
\item $R$ has a unique basic module which gives a splitting NCCR. 
\item $R$ has a steady splitting NCCR.
\item There exists a finite subgroup $G$ of $\Cl(R)$ such that $\bigoplus_{X\in G}X$ gives an NCCR of $R$.
\item $\Cl(R)$ is a finite group and $\bigoplus_{X\in\Cl(R)}X$ gives an NCCR of $R$.
\end{enumerate}
\end{theorem} 

When $R$ is a toric singularity, we obtain the following simple equivalent condition by applying results in toric geometry. 

\begin{corollary}
\label{toric} 
Let $R$ be a completion of a toric singularity over an algebraically closed field of characteristic zero. 
Then all the conditions in Theorem~\ref{main_intro} and the following condition are equivalent. 
\begin{enumerate}[$\bullet$]
\setlength{\parskip}{0pt} 
\setlength{\itemsep}{0pt}
\item $\Cl(R)$ is a finite group.
\end{enumerate}
\end{corollary}

Now we apply this result to dimer models.

Let us recall basic facts on dimer models. For more details, see Example~\ref{splitting_example} (b) and references therein.
A dimer model is a polygonal cell decomposition of the two-torus whose vertices and edges form a finite bipartite graph, 
and we can obtain a quiver with the potential $(Q_{\Gamma},W_{\Gamma})$ as the dual of a dimer model $\Gamma$. 
From a quiver with the potential $(Q_{\Gamma},W_{\Gamma})$, we define the complete Jacobian algebra $\mathcal{P}(Q_{\Gamma},W_{\Gamma})$. 
Under a certain condition called ``consistency condition", the center $R$ of $\mathcal{P}(Q_{\Gamma},W_{\Gamma})$ is a complete local Gorenstein toric singularity in dimension three, 
and $\mathcal{P}(Q_{\Gamma},W_{\Gamma})$ is a splitting NCCR of $R$. 

Thanks to our Theorem \ref{main_intro}, we have the following result which
characterizes when $\mathcal{P}(Q_{\Gamma},W_{\Gamma})$ is a steady NCCR of $R$.

\begin{corollary}\label{dimer}
Let $\Gamma$ be a consistent dimer model, $k$ an algebraically closed field of characteristic zero and $R$ the corresponding complete local Gorenstein toric singularity in dimension three. Then the following conditions are equivalent.
\begin{enumerate}[$\bullet$]
\setlength{\parskip}{0pt} 
\setlength{\itemsep}{0pt}
\item $\Gamma$ is homotopy equivalent to a regular hexagonal dimer model (i.e. each face of a dimer model is a regular hexagon).
\item $\Gamma$ gives a steady NCCR of $R$.
\item $R$ is a quotient singularity associated with a finite abelian group $G\subset\SL(3, k)$ (i.e. $R=S^G$ where $S=k[[x_1,x_2, x_3]]$).
\end{enumerate}
\end{corollary}

We show some examples of dimer models.

\begin{example}
The following figures are a consistent dimer model which is homotopy equivalent to a regular hexagonal dimer model, and the associated quiver. 
This quiver is just the McKay quiver of $G=\langle\mathrm{diag}(\omega,\omega^5,\omega^8)\rangle$ where 
$\omega$ is a primitive $14$-th root of unity. 
Also, the center of the Jacobian algebra is the quotient singularity associated with $G$. 

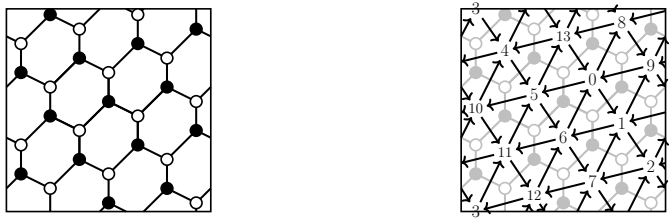
\begin{center}
\begin{tikzpicture}
\node (DM) at (0,0) 
{\scalebox{0.384}{
\begin{tikzpicture}
\node (B1) at (3,0){$$}; \node (B2) at (5,0.5){$$}; \node (B3) at (0,1){$$}; \node (B4) at (2,1.5){$$}; \node (B5) at (4,2){$$}; \node (B6) at (6,2.5){$$}; 
\node (B7) at (1,3){$$}; \node (B8) at (3,3.5){$$}; \node (B9) at (5,4){$$}; \node (B10) at (0,4.5){$$}; \node (B11) at (2,5){$$}; \node (B12) at (4,5.5){$$}; 
\node (B13) at (6,6){$$}; \node (B14) at (1,6.5){$$}; 
\node (W1) at (6,0){$$}; \node (W2) at (1,0.5){$$}; \node (W3) at (3,1){$$}; \node (W4) at (5,1.5){$$}; \node (W5) at (0,2){$$}; \node (W6) at (2,2.5){$$}; 
\node (W7) at (4,3){$$}; \node (W8) at (6,3.5){$$}; \node (W9) at (1,4){$$}; \node (W10) at (3,4.5){$$}; \node (W11) at (5,5){$$}; \node (W12) at (0,5.5){$$}; 
\node (W13) at (2,6){$$}; \node (W14) at (4,6.5){$$}; 
\draw[ultra thick]  (-0.5,-0.3) rectangle (6.5,6.7);

\draw[line width=0.07cm]  (W2)--(B4)--(W6)--(B7)--(W5)--(B3)--(W2); \draw[line width=0.07cm]  (W3)--(B5)--(W7)--(B8)--(W6)--(B4)--(W3);
\draw[line width=0.07cm]  (W4)--(B6)--(W8)--(B9)--(W7)--(B5)--(W4); \draw[line width=0.07cm]  (W6)--(B8)--(W10)--(B11)--(W9)--(B7)--(W6);
\draw[line width=0.07cm]  (W7)--(B9)--(W11)--(B12)--(W10)--(B8)--(W7); \draw[line width=0.07cm]  (W9)--(B11)--(W13)--(B14)--(W12)--(B10)--(W9);
\draw[line width=0.07cm]  (W3)--(B1); \draw[line width=0.07cm]  (W4)--(B2)--(W1); \draw[line width=0.07cm]  (B12)--(W14); \draw[line width=0.07cm]  (W11)--(B13);
\draw[line width=0.07cm]  (W2)--(1,-0.3); \draw[line width=0.07cm]  (W1)--(6,-0.3); \draw[line width=0.07cm]  (B13)--(6,6.7); 
\draw[line width=0.07cm]  (W1)--(6.5,0.5); \draw[line width=0.07cm]  (W8)--(6.5,4);\draw[line width=0.07cm]  (W5)--(-0.5,2.25); \draw[line width=0.07cm]  (W12)--(-0.5,5.75); 
\draw[line width=0.07cm]  (B3)--(-0.5,0.5); \draw[line width=0.07cm]  (B10)--(-0.5,4); \draw[line width=0.07cm]  (B6)--(6.5,2.25); \draw[line width=0.07cm]  (B13)--(6.5,5.75);
\draw[line width=0.07cm]  (B1)--(2.7,-0.3); \draw[line width=0.07cm]  (B1)--(3.6,-0.3); \draw[line width=0.07cm]  (B2)--(4.2,-0.3);
\draw[line width=0.07cm]  (W13)--(2.7,6.7); \draw[line width=0.07cm]  (W14)--(3.6,6.7); \draw[line width=0.07cm]  (W14)--(4.2,6.7);
\filldraw  [ultra thick, fill=black] (3,0) circle [radius=0.2] ; \filldraw  [ultra thick, fill=black] (5,0.5) circle [radius=0.2] ; 
\filldraw  [ultra thick, fill=black] (0,1) circle [radius=0.2] ; \filldraw  [ultra thick, fill=black] (2,1.5) circle [radius=0.2] ; \filldraw  [ultra thick, fill=black] (4,2) circle [radius=0.2] ;
\filldraw  [ultra thick, fill=black] (6,2.5) circle [radius=0.2] ; \filldraw  [ultra thick, fill=black] (1,3) circle [radius=0.2] ; \filldraw  [ultra thick, fill=black] (3,3.5) circle [radius=0.2] ;
\filldraw  [ultra thick, fill=black] (5,4) circle [radius=0.2] ; \filldraw  [ultra thick, fill=black] (0,4.5) circle [radius=0.2] ; \filldraw  [ultra thick, fill=black] (2,5) circle [radius=0.2] ;
\filldraw  [ultra thick, fill=black] (4,5.5) circle [radius=0.2] ; \filldraw  [ultra thick, fill=black] (6,6) circle [radius=0.2] ; \filldraw  [ultra thick, fill=black] (1,6.5) circle [radius=0.2] ;
\draw  [ultra thick,fill=white] (6,0) circle [radius=0.2] ; \draw  [ultra thick,fill=white] (1,0.5) circle [radius=0.2] ; \draw  [ultra thick,fill=white] (3,1) circle [radius=0.2] ;
\draw  [ultra thick,fill=white] (5,1.5) circle [radius=0.2] ; \draw  [ultra thick,fill=white] (0,2) circle [radius=0.2] ; \draw  [ultra thick,fill=white] (2,2.5) circle [radius=0.2] ; 
\draw  [ultra thick,fill=white] (4,3) circle [radius=0.2] ; \draw  [ultra thick,fill=white] (6,3.5) circle [radius=0.2] ; \draw  [ultra thick,fill=white] (1,4) circle [radius=0.2] ;
\draw  [ultra thick,fill=white] (3,4.5) circle [radius=0.2] ; \draw  [ultra thick,fill=white] (5,5) circle [radius=0.2] ; \draw  [ultra thick,fill=white] (0,5.5) circle [radius=0.2] ;
\draw  [ultra thick,fill=white] (2,6) circle [radius=0.2] ; \draw  [ultra thick,fill=white] (4,6.5) circle [radius=0.2] ;
\end{tikzpicture}
} }; 

\node (QV) at (6,0) 
{\scalebox{0.384}{
\begin{tikzpicture}
\node (B1) at (3,0){$$}; \node (B2) at (5,0.5){$$}; \node (B3) at (0,1){$$}; \node (B4) at (2,1.5){$$}; \node (B5) at (4,2){$$}; \node (B6) at (6,2.5){$$}; 
\node (B7) at (1,3){$$}; \node (B8) at (3,3.5){$$}; \node (B9) at (5,4){$$}; \node (B10) at (0,4.5){$$}; \node (B11) at (2,5){$$}; \node (B12) at (4,5.5){$$}; 
\node (B13) at (6,6){$$}; \node (B14) at (1,6.5){$$}; 
\node (W1) at (6,0){$$}; \node (W2) at (1,0.5){$$}; \node (W3) at (3,1){$$}; \node (W4) at (5,1.5){$$}; \node (W5) at (0,2){$$}; \node (W6) at (2,2.5){$$}; 
\node (W7) at (4,3){$$}; \node (W8) at (6,3.5){$$}; \node (W9) at (1,4){$$}; \node (W10) at (3,4.5){$$}; \node (W11) at (5,5){$$}; \node (W12) at (0,5.5){$$}; 
\node (W13) at (2,6){$$}; \node (W14) at (4,6.5){$$}; 

\node (Q0) at (4,4.25){{\huge$0$}}; \node (Q1) at (5,2.75){{\huge$1$}}; \node (Q2) at (6,1.25){{\huge$2$}}; \node (Q3) at (0,-0.3){{\huge$3$}}; 
\node (Q3a) at (0,6.7){{\huge$3$}};
\node (Q4) at (1,5.25){{\huge$4$}}; \node (Q5) at (2,3.75){{\huge$5$}}; \node (Q6) at (3,2.25){{\huge$6$}}; \node (Q7) at (4,0.75){{\huge$7$}}; 
\node (Q8) at (5,6.25){{\huge$8$}}; \node (Q9) at (6,4.75){{\huge$9$}}; \node (Q10) at (0,3.25){{\LARGE$10$}}; \node (Q11) at (1,1.75){{\LARGE$11$}}; 
\node (Q12) at (2,0.25){{\LARGE$12$}}; \node (Q13) at (3,5.75){{\LARGE$13$}}; 
\draw[lightgray, line width=0.07cm]  (W2)--(B4)--(W6)--(B7)--(W5)--(B3)--(W2); \draw[lightgray, line width=0.07cm]  (W3)--(B5)--(W7)--(B8)--(W6)--(B4)--(W3);
\draw[lightgray, line width=0.07cm]  (W4)--(B6)--(W8)--(B9)--(W7)--(B5)--(W4); \draw[lightgray, line width=0.07cm]  (W6)--(B8)--(W10)--(B11)--(W9)--(B7)--(W6);
\draw[lightgray, line width=0.07cm]  (W7)--(B9)--(W11)--(B12)--(W10)--(B8)--(W7); \draw[lightgray, line width=0.07cm]  (W9)--(B11)--(W13)--(B14)--(W12)--(B10)--(W9);
\draw[lightgray, line width=0.07cm]  (W3)--(B1); \draw[lightgray, line width=0.07cm]  (W4)--(B2)--(W1); 
\draw[lightgray, line width=0.07cm]  (B12)--(W14); \draw[lightgray, line width=0.07cm]  (W11)--(B13);
\draw[lightgray, line width=0.07cm]  (W2)--(1,-0.3); \draw[lightgray, line width=0.07cm]  (W1)--(6,-0.3); 
\draw[lightgray, line width=0.07cm]  (B13)--(6,6.7); 
\draw[lightgray, line width=0.07cm]  (W1)--(6.5,0.5); \draw[lightgray, line width=0.07cm]  (W8)--(6.5,4);
\draw[lightgray, line width=0.07cm]  (W5)--(-0.5,2.25); \draw[lightgray, line width=0.07cm]  (W12)--(-0.5,5.75); 
\draw[lightgray, line width=0.07cm]  (B3)--(-0.5,0.5); \draw[lightgray, line width=0.07cm]  (B10)--(-0.5,4); 
\draw[lightgray, line width=0.07cm]  (B6)--(6.5,2.25); \draw[lightgray, line width=0.07cm]  (B13)--(6.5,5.75);
\draw[lightgray, line width=0.07cm]  (B1)--(2.7,-0.3); \draw[lightgray, line width=0.07cm]  (B1)--(3.6,-0.3); \draw[lightgray, line width=0.07cm]  (B2)--(4.2,-0.3);
\draw[lightgray, line width=0.07cm]  (W13)--(2.7,6.7); \draw[lightgray, line width=0.07cm]  (W14)--(3.6,6.7); \draw[lightgray, line width=0.07cm]  (W14)--(4.2,6.7);
\filldraw  [ultra thick, draw=lightgray, fill=lightgray] (3,0) circle [radius=0.2] ; \filldraw  [ultra thick, draw=lightgray, fill=lightgray] (5,0.5) circle [radius=0.2] ; 
\filldraw  [ultra thick, draw=lightgray, fill=lightgray] (0,1) circle [radius=0.2] ; \filldraw  [ultra thick, draw=lightgray, fill=lightgray] (2,1.5) circle [radius=0.2] ; 
\filldraw  [ultra thick, draw=lightgray, fill=lightgray] (4,2) circle [radius=0.2] ;
\filldraw  [ultra thick, draw=lightgray, fill=lightgray] (6,2.5) circle [radius=0.2] ; \filldraw  [ultra thick, draw=lightgray, fill=lightgray] (1,3) circle [radius=0.2] ; 
\filldraw  [ultra thick, draw=lightgray, fill=lightgray] (3,3.5) circle [radius=0.2] ;
\filldraw  [ultra thick, draw=lightgray, fill=lightgray] (5,4) circle [radius=0.2] ; \filldraw  [ultra thick, draw=lightgray, fill=lightgray] (0,4.5) circle [radius=0.2] ; 
\filldraw  [ultra thick, draw=lightgray, fill=lightgray] (2,5) circle [radius=0.2] ;
\filldraw  [ultra thick, draw=lightgray, fill=lightgray] (4,5.5) circle [radius=0.2] ; \filldraw  [ultra thick, draw=lightgray, fill=lightgray] (6,6) circle [radius=0.2] ; 
\filldraw  [ultra thick, draw=lightgray, fill=lightgray] (1,6.5) circle [radius=0.2] ;
\draw  [ultra thick,draw=lightgray,fill=white] (6,0) circle [radius=0.2] ; \draw  [ultra thick,draw=lightgray,fill=white] (1,0.5) circle [radius=0.2] ; 
\draw  [ultra thick,draw=lightgray,fill=white] (3,1) circle [radius=0.2] ;
\draw  [ultra thick,draw=lightgray,fill=white] (5,1.5) circle [radius=0.2] ; \draw  [ultra thick,draw=lightgray,fill=white] (0,2) circle [radius=0.2] ; 
\draw  [ultra thick,draw=lightgray,fill=white] (2,2.5) circle [radius=0.2] ; 
\draw  [ultra thick,draw=lightgray,fill=white] (4,3) circle [radius=0.2] ; \draw  [ultra thick,draw=lightgray,fill=white] (6,3.5) circle [radius=0.2] ; 
\draw  [ultra thick,draw=lightgray,fill=white] (1,4) circle [radius=0.2] ;
\draw  [ultra thick,draw=lightgray,fill=white] (3,4.5) circle [radius=0.2] ; \draw  [ultra thick,draw=lightgray,fill=white] (5,5) circle [radius=0.2] ; 
\draw  [ultra thick,draw=lightgray,fill=white] (0,5.5) circle [radius=0.2] ;
\draw  [ultra thick,draw=lightgray,fill=white] (2,6) circle [radius=0.2] ; \draw  [ultra thick,draw=lightgray,fill=white] (4,6.5) circle [radius=0.2] ;

\draw[->, line width=0.07cm] (Q0)--(Q1); \draw[->, line width=0.07cm] (Q1)--(Q2); \draw[->, line width=0.07cm] (Q4)--(Q5); 
\draw[->, line width=0.07cm] (Q5)--(Q6); \draw[->, line width=0.07cm] (Q6)--(Q7); \draw[->, line width=0.07cm] (Q13)--(Q0);
\draw[->, line width=0.07cm] (Q8)--(Q9); \draw[->, line width=0.07cm] (Q10)--(Q11); \draw[->, line width=0.07cm] (Q11)--(Q12);
\draw[->, line width=0.07cm] (Q8)--(Q13); \draw[->, line width=0.07cm] (Q13)--(Q4); \draw[->, line width=0.07cm] (Q9)--(Q0); 
\draw[->, line width=0.07cm] (Q0)--(Q5); \draw[->, line width=0.07cm] (Q5)--(Q10); \draw[->, line width=0.07cm] (Q1)--(Q6); 
\draw[->, line width=0.07cm] (Q6)--(Q11); \draw[->, line width=0.07cm] (Q2)--(Q7); \draw[->, line width=0.07cm] (Q7)--(Q12); 
\draw[->, line width=0.07cm] (Q12)--(Q3); \draw[->, line width=0.07cm] (Q3)--(Q11); \draw[->, line width=0.07cm] (Q11)--(Q5); 
\draw[->, line width=0.07cm] (Q5)--(Q13); \draw[->, line width=0.07cm] (Q12)--(Q6); \draw[->, line width=0.07cm] (Q6)--(Q0); 
\draw[->, line width=0.07cm] (Q0)--(Q8); \draw[->, line width=0.07cm] (Q10)--(Q4); \draw[->, line width=0.07cm] (Q7)--(Q1);
\draw[->, line width=0.07cm] (Q1)--(Q9);
\draw[->, line width=0.07cm] (6.5,1.375)--(Q2); \draw[->, line width=0.07cm] (Q11)--(-0.5,1.375); 
\draw[->, line width=0.07cm] (6.5,6.625)--(Q8); \draw[->, line width=0.07cm] (Q4)--(-0.5,4.875);
\draw[->, line width=0.07cm] (6.5,4.875)--(Q9); \draw[->, line width=0.07cm] (6.5,3.125)--(Q1); \draw[->, line width=0.07cm] (Q10)--(-0.5,3.125);
\draw[->, line width=0.07cm] (Q12)--(2.3666,-0.3); \draw[->, line width=0.07cm] (2.3666,6.7)--(Q13);
\draw[->, line width=0.07cm] (Q7)--(4.7,-0.3); \draw[->, line width=0.07cm] (4.7,6.7)--(Q8); \draw[->, line width=0.07cm] (Q3a)--(Q4);
\draw[->, line width=0.07cm] (Q9)--(6.5,4); \draw[->, line width=0.07cm] (-0.5,4)--(Q10); \draw[->, line width=0.07cm] (Q2)--(6.5,0.5); 
\draw[->, line width=0.07cm] (-0.5,0.5)--(Q3); \draw[->, line width=0.07cm] (Q4)--(1.725,6.7); \draw[->, line width=0.07cm] (1.725,-0.3)--(Q12);
\draw[->, line width=0.07cm] (3.475,-0.3)--(Q7); \draw[->, line width=0.07cm] (Q13)--(3.475,6.7); 
\draw[->, line width=0.07cm] (5.225,-0.3)--(Q2); \draw[->, line width=0.07cm] (Q2)--(6.5,2.25); \draw[->, line width=0.07cm] (-0.5,2.25)--(Q10);
\draw[->, line width=0.07cm] (Q9)--(6.5,5.75); \draw[->, line width=0.07cm] (-0.5,5.75)--(Q3a); 

\draw[ultra thick]  (-0.5,-0.3) rectangle (6.5,6.7);
\end{tikzpicture}
} }; 
\end{tikzpicture}
\end{center}

On the other hand, the following consistent dimer model is not homotopy equivalent to a regular hexagonal dimer model, 
and it gives the toric singularity defined by the cone $\sigma$: 
\[ \sigma=\mathrm{Cone}\{(1,1,1), (-1,1,1), (-1,-1,1), (1,-1,1) \}. \]

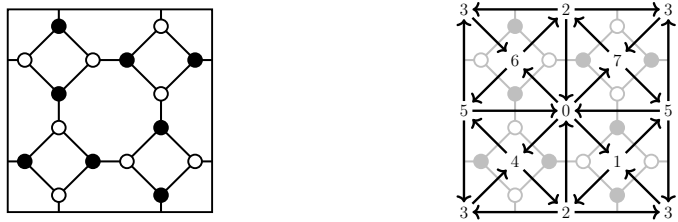
\begin{center}
\begin{tikzpicture}
\node (DM) at (0,0) 
{\scalebox{0.448}{
\begin{tikzpicture}
\node (B1) at (0.5,1.5){$$}; \node (B2) at (2.5,1.5){$$}; \node (B3) at (4.5,0.5){$$};  \node (B4) at (4.5,2.5){$$};
\node (B5) at (1.5,3.5){$$};  \node (B6) at (1.5,5.5){$$}; \node (B7) at (3.5,4.5){$$};  \node (B8) at (5.5,4.5){$$};
\node (W1) at (1.5,0.5){$$}; \node (W2) at (1.5,2.5){$$}; \node (W3) at (3.5,1.5){$$};  \node (W4) at (5.5,1.5){$$};
\node (W5) at (0.5,4.5){$$}; \node (W6) at (2.5,4.5){$$}; \node (W7) at (4.5,3.5){$$}; \node (W8) at (4.5,5.5){$$};
\draw[ultra thick]  (0,0) rectangle (6,6);
\draw[line width=0.06cm] (B1)--(W1); \draw[line width=0.06cm] (B1)--(W2); \draw[line width=0.06cm] (B2)--(W1); 
\draw[line width=0.06cm] (B2)--(W2); \draw[line width=0.06cm] (B2)--(W3); \draw[line width=0.06cm] (B3)--(W3);
\draw[line width=0.06cm] (B3)--(W4); \draw[line width=0.06cm] (B4)--(W3); \draw[line width=0.06cm] (B4)--(W4); 
\draw[line width=0.06cm] (B4)--(W7); \draw[line width=0.06cm] (B5)--(W2); \draw[line width=0.06cm] (B5)--(W5); 
\draw[line width=0.06cm] (B5)--(W6); \draw[line width=0.06cm] (B6)--(W5); \draw[line width=0.06cm] (B6)--(W6); 
\draw[line width=0.06cm] (B7)--(W6); \draw[line width=0.06cm] (B7)--(W7); \draw[line width=0.06cm] (B7)--(W8); 
\draw[line width=0.06cm] (B8)--(W7); \draw[line width=0.06cm] (B8)--(W8); 
\draw[line width=0.06cm] (B1)--(0,1.5); \draw[line width=0.06cm] (W4)--(6,1.5); \draw[line width=0.06cm] (B6)--(1.5,6); 
\draw[line width=0.06cm] (W1)--(1.5,0); \draw[line width=0.06cm] (B3)--(4.5,0); \draw[line width=0.06cm] (W8)--(4.5,6); 
\draw[line width=0.06cm] (B8)--(6,4.5); \draw[line width=0.06cm] (W5)--(0,4.5); 

\filldraw  [ultra thick, fill=black] (0.5,1.5) circle [radius=0.2] ;\filldraw  [ultra thick, fill=black] (2.5,1.5) circle [radius=0.2] ;
\filldraw  [ultra thick, fill=black] (4.5,0.5) circle [radius=0.2] ;\filldraw  [ultra thick, fill=black] (4.5,2.5) circle [radius=0.2] ;
\filldraw  [ultra thick, fill=black] (1.5,3.5) circle [radius=0.2] ;\filldraw  [ultra thick, fill=black] (1.5,5.5) circle [radius=0.2] ;
\filldraw  [ultra thick, fill=black] (3.5,4.5) circle [radius=0.2] ;\filldraw  [ultra thick, fill=black] (5.5,4.5) circle [radius=0.2] ;
\draw  [ultra thick,fill=white] (1.5,0.5) circle [radius=0.2] ;\draw  [ultra thick, fill=white] (1.5,2.5) circle [radius=0.2] ;
\draw  [ultra thick,fill=white] (3.5,1.5)circle [radius=0.2] ;\draw  [ultra thick, fill=white] (5.5,1.5) circle [radius=0.2] ;
\draw  [ultra thick,fill=white] (0.5,4.5)circle [radius=0.2] ;\draw  [ultra thick, fill=white] (2.5,4.5) circle [radius=0.2] ; 
\draw  [ultra thick,fill=white] (4.5,3.5)circle [radius=0.2] ;\draw  [ultra thick, fill=white] (4.5,5.5) circle [radius=0.2] ;
\end{tikzpicture}
} }; 

\node (QV) at (6,0) 
{\scalebox{0.448}{
\begin{tikzpicture}
\node (B1) at (0.5,1.5){$$}; \node (B2) at (2.5,1.5){$$}; \node (B3) at (4.5,0.5){$$};  \node (B4) at (4.5,2.5){$$};
\node (B5) at (1.5,3.5){$$};  \node (B6) at (1.5,5.5){$$}; \node (B7) at (3.5,4.5){$$};  \node (B8) at (5.5,4.5){$$};
\node (W1) at (1.5,0.5){$$}; \node (W2) at (1.5,2.5){$$}; \node (W3) at (3.5,1.5){$$};  \node (W4) at (5.5,1.5){$$};
\node (W5) at (0.5,4.5){$$}; \node (W6) at (2.5,4.5){$$}; \node (W7) at (4.5,3.5){$$}; \node (W8) at (4.5,5.5){$$};

\node (Q0) at (4.5,1.5){{\LARGE$1$}}; \node (Q1a) at (0,0){{\LARGE$3$}}; \node (Q1b) at (6,0){{\LARGE$3$}}; 
\node (Q1c) at (6,6){{\LARGE$3$}}; \node (Q1d) at (0,6){{\LARGE$3$}}; \node (Q2) at (1.5,1.5){{\LARGE$4$}}; 
\node (Q3a) at (0,3){{\LARGE$5$}}; \node (Q3b) at (6,3){{\LARGE$5$}}; \node (Q4) at (1.5,4.5){{\LARGE$6$}};
\node (Q5) at (3,3){{\LARGE$0$}}; \node (Q6) at (4.5,4.5){{\LARGE$7$}}; \node (Q7a) at (3,0){{\LARGE$2$}}; \node (Q7b) at (3,6){{\LARGE$2$}};

\draw[lightgray, line width=0.06cm] (B1)--(W1); \draw[lightgray, line width=0.06cm] (B1)--(W2); \draw[lightgray, line width=0.06cm] (B2)--(W1); 
\draw[lightgray, line width=0.06cm] (B2)--(W2); \draw[lightgray, line width=0.06cm] (B2)--(W3); \draw[lightgray, line width=0.06cm] (B3)--(W3);
\draw[lightgray, line width=0.06cm] (B3)--(W4); \draw[lightgray, line width=0.06cm] (B4)--(W3); \draw[lightgray, line width=0.06cm] (B4)--(W4); 
\draw[lightgray, line width=0.06cm] (B4)--(W7); \draw[lightgray, line width=0.06cm] (B5)--(W2); \draw[lightgray, line width=0.06cm] (B5)--(W5); 
\draw[lightgray, line width=0.06cm] (B5)--(W6); \draw[lightgray, line width=0.06cm] (B6)--(W5); \draw[lightgray, line width=0.06cm] (B6)--(W6); 
\draw[lightgray, line width=0.06cm] (B7)--(W6); \draw[lightgray, line width=0.06cm] (B7)--(W7); \draw[lightgray, line width=0.06cm] (B7)--(W8); 
\draw[lightgray, line width=0.06cm] (B8)--(W7); \draw[lightgray, line width=0.06cm] (B8)--(W8); 
\draw[lightgray, line width=0.06cm] (B1)--(0,1.5); \draw[lightgray, line width=0.06cm] (W4)--(6,1.5); \draw[lightgray, line width=0.06cm] (B6)--(1.5,6); 
\draw[lightgray, line width=0.06cm] (W1)--(1.5,0); \draw[lightgray, line width=0.06cm] (B3)--(4.5,0); \draw[lightgray, line width=0.06cm] (W8)--(4.5,6); 
\draw[lightgray, line width=0.06cm] (B8)--(6,4.5); \draw[lightgray, line width=0.06cm] (W5)--(0,4.5); 
\filldraw  [ultra thick, draw=lightgray, fill=lightgray] (0.5,1.5) circle [radius=0.2] ;\filldraw  [ultra thick, draw=lightgray, fill=lightgray] (2.5,1.5) circle [radius=0.2] ;
\filldraw  [ultra thick, draw=lightgray, fill=lightgray] (4.5,0.5) circle [radius=0.2] ;\filldraw  [ultra thick, draw=lightgray, fill=lightgray] (4.5,2.5) circle [radius=0.2] ;
\filldraw  [ultra thick, draw=lightgray, fill=lightgray] (1.5,3.5) circle [radius=0.2] ;\filldraw  [ultra thick, draw=lightgray, fill=lightgray] (1.5,5.5) circle [radius=0.2] ;
\filldraw  [ultra thick, draw=lightgray, fill=lightgray] (3.5,4.5) circle [radius=0.2] ;\filldraw  [ultra thick, draw=lightgray, fill=lightgray] (5.5,4.5) circle [radius=0.2] ;
\draw  [ultra thick,draw=lightgray,fill=white] (1.5,0.5) circle [radius=0.2] ;\draw  [ultra thick, draw=lightgray,fill=white] (1.5,2.5) circle [radius=0.2] ;
\draw  [ultra thick,draw=lightgray,fill=white] (3.5,1.5)circle [radius=0.2] ;\draw  [ultra thick, draw=lightgray,fill=white] (5.5,1.5) circle [radius=0.2] ;
\draw  [ultra thick,draw=lightgray,fill=white] (0.5,4.5)circle [radius=0.2] ;\draw  [ultra thick, draw=lightgray,fill=white] (2.5,4.5) circle [radius=0.2] ; 
\draw  [ultra thick,draw=lightgray,fill=white] (4.5,3.5)circle [radius=0.2] ;\draw  [ultra thick, draw=lightgray,fill=white] (4.5,5.5) circle [radius=0.2] ;
\draw[->, line width=0.07cm] (Q0)--(Q3b); \draw[->, line width=0.07cm] (Q0)--(Q7a); \draw[->, line width=0.07cm] (Q1a)--(Q2); 
\draw[->, line width=0.07cm] (Q1b)--(Q0); \draw[->, line width=0.07cm] (Q1c)--(Q6); \draw[->, line width=0.07cm] (Q1d)--(Q4); 
\draw[->, line width=0.07cm] (Q2)--(Q3a); \draw[->, line width=0.07cm] (Q2)--(Q7a); \draw[->, line width=0.07cm] (Q3a)--(Q5); 
\draw[->, line width=0.07cm] (Q3a)--(Q1a); \draw[->, line width=0.07cm] (Q3a)--(Q1d); \draw[->, line width=0.07cm] (Q3b)--(Q5); 
\draw[->, line width=0.07cm] (Q3b)--(Q1b); \draw[->, line width=0.07cm] (Q3b)--(Q1c); 
\draw[->, line width=0.07cm] (Q4)--(Q3a); \draw[->, line width=0.07cm] (Q4)--(Q7b); \draw[->, line width=0.07cm] (Q5)--(Q0); 
\draw[->, line width=0.07cm] (Q5)--(Q2); \draw[->, line width=0.07cm] (Q5)--(Q4); \draw[->, line width=0.07cm] (Q5)--(Q6); 
\draw[->, line width=0.07cm] (Q6)--(Q3b); \draw[->, line width=0.07cm] (Q6)--(Q7b); \draw[->, line width=0.07cm] (Q7a)--(Q1a); 
\draw[->, line width=0.07cm] (Q7a)--(Q1b); \draw[->, line width=0.07cm] (Q7a)--(Q5); \draw[->, line width=0.07cm] (Q7b)--(Q1c); 
\draw[->, line width=0.07cm] (Q7b)--(Q1d); \draw[->, line width=0.07cm] (Q7b)--(Q5); 
\end{tikzpicture}
} }; 
\end{tikzpicture}
\end{center}
\end{example}

\medskip 

In what follows, we will investigate properties of steady NCCRs, and those of splitting NCCRs in Section~\ref{sec_steadyNCCR}. 
In Section~\ref{sec_main}, we will give a proof of Theorem~\ref{main_intro} in a more detailed form. 

\subsection*{Notations and Conventions} 
Let $R$ be a commutative Noetherian ring. 
We denote $\Mod R$ to be the category of $R$-modules, $\mod R$ to be the category of finitely generated $R$-modules, 
$\add_RM$ to be the full subcategory consisting of direct summands of finite direct sums of some copies of $M\in \mod R$. 
When $R$ is $G$-graded for an abelian group $G$, we denote by $\Mod^GR$ the category of $G$-graded $R$-module, 
and $\mod^GR$ the category of finitely generated $G$-graded $R$-module.
We say an $R$-module $M=M_1\oplus\cdots\oplus M_n$ is \emph{basic} if $M_i$'s are mutually non-isomorphic. 
Also, we denote by $\Cl(R)$ the class group of $R$. 
When we consider a rank one reflexive $R$-module $I$ as an element of $\Cl(R)$, we denote it by $[I]$. 

\subsection*{Acknowledgements}
The authors thank Ragnar-Olaf Buchweitz for stimulating discussions.
The first author thanks Akira Ishii and Kazushi Ueda for kind explanations on dimer models. 
The first author was partially supported by JSPS Grant-in-Aid for Scientific Research (B) 16H03923, (C) 23540045 and (S) 15H05738. 
The second author was supported by Grant-in-Aid for JSPS Fellows No. 26--422. 
Both authors thank the anonymous referee for valuable comments on the paper. 

\section{Steady NCCRs and Splitting NCCRs} 
\label{sec_steadyNCCR}

\subsection{Preliminaries}  
We start with preparing basic facts used in this paper.
We denote the $R$-dual functor by 
\[(-)^*:=\Hom_R(-, R) : \mod R\rightarrow \mod R.\] 
We say that $M\in \mod R$ is \emph{reflexive} if the natural morphism $M\rightarrow M^{**}$ is an isomorphism. 
We denote $\refl R$ to be the category of reflexive $R$-modules. 

Let $(R, \mm)$ be a commutative Noetherian local ring. 
For $M\in\mod R$, we define the depth of $M$ as 
\[
\depth_RM:=\mathrm{inf} \{ i\ge0 \mid \Ext^i_R(R/\mm, M)\neq 0\}. 
\]
We say $M$ is a \emph{maximal Cohen-Macaulay} (or \emph{CM} for short) $R$-module if $\depth_RM={\rm dim}R$. 
When $R$ is non-local, we say $M$ is a CM module if $M_\pp$ is a CM module for all $\pp\in\Spec R$. 
Furthermore, we say that $R$ is a \emph{Cohen-Macaulay ring} (= \emph{CM ring}) if $R$ is a CM $R$-module. 
We denote $\CM R$ to be the category of Cohen-Macaulay $R$-modules. 

\medskip  

The following is well-known. For example, see \cite[II. 2.1]{ARS}. 

\begin{lemma}
\label{projectivization}
Let $R$ be a commutative ring and $M\in\mod R$. Suppose $M$ is a generator. Then  
the functor $\Hom_R(M,-):\mod R\rightarrow \mod \End_R(M)$ is fully faithful, restricting to an equivalence 
$\add_RM\simeq\proj\End_R(M)$.
\end{lemma}

Also, the following property of modules giving NCCRs is important.

\begin{proposition}(see \cite[4.5]{IW2})
\label{maximal_NCCR}
Let $R$ be a $d$-dimensional normal CM local ring. 
Suppose $M\in\refl R$ gives an NCCR of $R$. Then for any $X\in \refl R$ with $\End_R(M\oplus X)\in\CM R$, 
we have $X\in\add_RM$. 
\end{proposition}


\subsection{Examples}

We provide examples of steady NCCRs and those of splitting NCCRs.

\begin{example}
\label{steady_example} 
Let $S=k[[x_1,\cdots,x_d]]$ be a formal power series ring, $G$ be a finite subgroup of $\GL(d, k)$ such that $|G|$ is invertible in $k$. 
Let $R:=S^G$. 
\begin{enumerate}[\rm(a)]
\setlength{\parskip}{0pt} 
\setlength{\itemsep}{0pt}
\item The $R$-module $S$ gives a steady NCCR. 
\item Assume that $G$ is small (i.e. it contains no pseudo-reflections except the identity).
Let $\{V_i\mid i\in I\}$ be the set of simple $kG$-modules. Then we have
$S\simeq\bigoplus_{i\in I}((S\otimes_{k}V_i)^G)^{\oplus\dim_kV_i}$ as $R$-modules,
where $(S\otimes_{k}V_i)^G$ is an indecomposable CM $R$-module with rank $\dim_kV_i$.
\end{enumerate}
\end{example} 

\begin{proof}
(i) We prove the statements under the assumption that $G$ is small. 

In this case, we have $\End_R(S)\simeq S*G$, where $S*G$ is the skew group ring \cite{Aus1} (see also \cite[4.2]{IT}, \cite[5.12]{LW}). 
Since $S*G$ is a non-singular $R$-order (see e.g. \cite[2.12]{IW2}) and $\End_R(S)\in\add_R S$ holds, we have the assertion (a).

We have $S=(S\otimes_{k}kG)^G=(S\otimes_{k}(\bigoplus_{i\in I}V_i^{\oplus\dim_kV_i}))^G=\bigoplus_{i\in I}((S\otimes_{k}V_i)^G)^{\oplus\dim_kV_i}$ as $R$-modules.
This is a decomposition into indecomposable $R$-modules since $\End_R(S)/{\rm rad} \End_R(S)\simeq S*G/{\rm rad}(S*G)=kG$ holds. Thus we have the assertion (b).

(ii) We prove the statement (a) for general case. In this case, there exists a formal power series ring  $T$ and a finite small subgroup $G^\prime\subset\GL(d,k)$ 
such that $R\subset T\subset S$, $R=T^{G^\prime}$, and $S$ is a free $T$-module of finite rank (see \cite[Proof of 5.7]{IW2}). 
Suppose $S\simeq T^{\oplus r}$. Then $\End_R(S)=M_r(\End_R(T))$ belongs to $\add_R T=\add_R S$. 
Therefore, the assertion follows from (i).
\end{proof}

Also, we provide Examples of splitting NCCRs. 

\begin{example}
\label{splitting_example}
\begin{enumerate}[\rm(a)]
\setlength{\parskip}{0pt} 
\setlength{\itemsep}{0pt}
\item Let $R=S^G$ be a quotient singularity by a finite abelian group $G\subset\GL(d, k)$ such that $|G|$ is invertible in $k$. In this situation, $S$ gives a steady splitting NCCR of $R$ by Example~\ref{steady_example}. 
\item Let $R$ be a Gorenstein toric singularity in dimension three. Then the method of dimer models gives splitting NCCRs of $R$ (see e.g. \cite{Bro, IU, Boc1}). 
Conversely any splitting NCCR of $R$ is given by a consistent dimer model \cite{Boc3}. 

\item Let $R=k[[x,y,u,v]]/(f(x,y)-uv)$ be a $cA_n$ singularity. If $f$ is a product $f_1\cdots f_m$ with $f_i\not\in (x,y)^2$, 
then $R$ has a splitting NCCR \cite{IW4}. 
\end{enumerate}
\end{example}

\subsection{Basic properties}
We show some basic properties of steady NCCRs and splitting NCCRs. 

\begin{lemma}
\label{basic_pro}
Let $M$ be a steady $R$-module.
\begin{enumerate}[\rm(a)]
\setlength{\parskip}{0pt} 
\setlength{\itemsep}{0pt}
\item $\add_RM=\add_R\End_R(M)$.
\item $M^*$ is a steady $R$-module and satisfies $\add_RM=\add_RM^*$.
\item If $M$ gives an NCCR, then $M\in\CM R$.
\end{enumerate}
\end{lemma}

\begin{proof}
(a) Since $M$ is a generator, we have $M\simeq\Hom_R(R,M)\in\add_R\End_R(M)$.

(b) We have $M^*\simeq\Hom_R(M,R)\in\add_R\End_R(M)=\add_RM$ by (a).
Similarly we have $M\simeq\Hom_R(M^*,R)\in\add_R\Hom_R(M,R)=\add_RM^*$.

(c) This is clear since $M\simeq\Hom_R(R,M)\in\add_R\End_R(M)\subset\CM R$.
\end{proof}

\begin{lemma}
\label{prop_steady}
Let $R$ be a CM normal domain. Suppose $M$ satisfies $\End_R(M)\in\add_R M$. 
Then $M$ is a generator if one of the following conditions is satisfied. 
  \begin{itemize}
  \setlength{\parskip}{0pt} 
  \setlength{\itemsep}{0pt}
   \item $R$ contains a field of characteristic zero, 
   \item $M$ has a rank one reflexive module as a direct summand. 
  \end{itemize}
In particular, if $M$ is splitting, then $M$ is a generator. 
\end{lemma}

\begin{proof}
Firstly, if $R$ contains a field of characteristic zero, we have $R\in\add_R\End_R(M)\subset\add_RM$ by \cite[5.6]{Aus3}. 
Next, we suppose that $I$ is a rank one reflexive $R$-module such that $I\in\add_R M$. Then we have $R\simeq\Hom_R(I, I)\in\add_R(M)$. 
\end{proof}

The following characterization of steady splitting modules is important.

\begin{proposition}
\label{steady_group} 
Let $R$ be a complete local normal domain, and let $\mathsf{M}$ be a finite subset of $\Cl(R)$ and $M\coloneqq\bigoplus_{X\in\mathsf{M}}X$. 
Then $M$ is steady if and only if $\mathsf{M}$ is a subgroup of $\Cl(R)$. 
\end{proposition}

\begin{proof}
Let $\mathsf{M}\coloneqq\{[M_1],\ldots,[M_n]\}$. 
Then $\End_R(M)\in\add_RM$ means that $[M_i]-[M_j]\in\mathsf{M}$ holds for any $i$ and $j$. 
This is clealy equivalent to that $\mathsf{M}$ forms a subgroup of $\Cl(R)$.
\end{proof}

We denote by ${\rm K}_0(R)$ the Grothendieck group of $R$. 
The following propositions play the crucial role to proof the main theorem. 

\begin{proposition}
\label{cl_finite}
Let $R$ be a complete local CM normal domain. 
Suppose $M=\bigoplus_{i=1}^nM_i$ gives a splitting NCR of $R$, and $M$ is a generator. 
\begin{enumerate}[\rm(a)]
\setlength{\parskip}{0pt} 
\setlength{\itemsep}{0pt}
\item ${\rm K}_0(R)$ and $\Cl(R)$ are generated by $[M_1],\ldots,[M_n]$.
\item If $M$ is steady, then $\Cl(R)=\{[M_1],\ldots,[M_n]\}$.
\end{enumerate}
\end{proposition}

\begin{proof}
\rm(a) 
We have a surjection ${\rm K}_0(R)\twoheadrightarrow \Cl(R)$ by \cite[VII.4.7]{Bou}. 
Thus, it is enough to show that ${\rm K}_0(R)$ is generated by
$[M_1],\cdots,[M_n]$.
Let $E:=\End_R(M)$. For any $X\in\mod R$, let $Y:=\Hom_R(M,X)\in\mod E$.
Since ${\rm gl.dim}\,E<\infty$, there exists an exact sequence
\[0\to P_r\to\cdots\to P_0\to Y\to0.\] 
Since $M$ is a generator, $P_i=\Hom_R(M,N_i)$ for some $N_i\in\add_RM$ and we have 
an exact sequence below (see Lemma~\ref{projectivization}), 
\[0\to N_r\to\cdots\to N_0\to X\to0.\] 
Thus $X$ belongs to the subgroup $\langle[M_1],\cdots,[M_n]\rangle$ of $\Cl(R)$.

\rm(b) 
The assertion follows from (a) and Proposition~\ref{steady_group}. 
\end{proof}

\section{Proof of the Main Theorem}
\label{sec_main}

In this section, we give a proof of the main theorem (see Theorem~\ref{main_thm}). 
To state the theorem, we prepare general observations on graded rings.
Let $G$ be a finite abelian group and $S=\bigoplus_{i\in G}S_i$ a $G$-graded ring. We say that $S$ is \emph{strongly $G$-graded} if the map $S_j\to\Hom_{S_0}(S_i,S_{i+j})$ sending $x\in S_j$ to $(y\mapsto yx)\in\Hom_R(S_i,S_{i+j})$ is an isomorphism for any $i,j\in G$.

\medskip

Now we restate the main theorem in a detailed form. 

\begin{theorem}
\label{main_thm}
Let $R$ be a $d$-dimensional complete local Cohen-Macaulay normal domain. 
Consider the following conditions.
\begin{enumerate}[\rm(1)]
\setlength{\parskip}{0pt} 
\setlength{\itemsep}{0pt}
\item $R$ is a quotient singularity associated with a finite (small) abelian group $G\subset\GL(d, k)$ (i.e. $R=S^G$ where $S=k[[x_1, \cdots, x_d]]$ for a field $k$).
\item There exist a finite abelian group $G$ and a complete regular local ring $S$ which is strongly $G$-graded such that $R=S_0$.
\item $R$ has a unique basic module which gives a splitting NCCR. 
\item $R$ has a steady splitting NCCR.
\item $R$ has a steady splitting NCR.
\item There exists a finite subgroup $G$ of $\Cl(R)$ such that $\bigoplus_{X\in G}X$ gives an NCCR of $R$.
\item There exists a finite subgroup $G$ of $\Cl(R)$ such that $\bigoplus_{X\in G}X$ gives an NCR of $R$.
\item $\Cl(R)$ is a finite group and $\bigoplus_{X\in\Cl(R)}X$ gives an NCCR of $R$.
\item $\Cl(R)$ is a finite group and $\bigoplus_{X\in\Cl(R)}X$ gives an NCR of $R$.
\end{enumerate}
 
Then the conditions $(2)$--$(9)$ are equivalent. 

If $R$ contains an algebraically closed field of characteristic zero, then all the conditions $(1)$--$(9)$ are equivalent.
\end{theorem} 

In order to prove this, we need some preparations.
Let $G$ be a finite abelian group and $S=\bigoplus_{i\in G}S_i$ a $G$-graded ring. Then the \emph{smash product} \cite{CM} 
\[ S\# G:=(S_{j-i})_{i,j\in G}\]
is a ring whose multiplication is given by the multiplication in $S$ and the matrix multiplication rule. 
We have a natural morphism
\begin{equation}\label{SG and End}
\phi:S\# G\to\End_R(S)
\end{equation}
of rings sending $(x_{ij})_{i,j\in G}\in S\# G$ to $(S\ni (y_j)_{j\in G}\mapsto (\sum_{i\in G}y_ix_{ij})_{j\in G}\in S)\in\End_R(S)$.
Clearly $S$ is strongly $G$-graded if and only if $\phi$ is an isomorphism.

We need the following observations to prove Theorem~\ref{main_thm}.

\begin{proposition}\label{prepare graded}
\begin{enumerate}[\rm(a)]
\setlength{\parskip}{0pt} 
\setlength{\itemsep}{0pt}
\item Let $G$ be a finite abelian group and $S$ a $G$-graded ring. Then we have an equivalence $\Mod^GS\simeq\Mod(S\# G)$, which preserves finitely generated modules.
\item Let $R$ be a normal domain and $G$ a finite subgroup of $\Cl(R)$. Then there exists a strongly $G$-graded ring $S$ such that $S_0=R$ and $[S_i]=i$ in $\Cl(R)$ for any $i\in G$.
Moreover, if $R$ is complete local, then so is $S$.
\item Let $G$ be a finite abelian group, $S$ a normal domain which is $G$-graded and $R=S_0$.
If $\Cl(S)=0$ holds and the $R$-module $S_i$ has rank one for any $i\in G$, then $\Cl(R)=\{[S_i]\mid i\in G\}$.
\end{enumerate}
\end{proposition}

\begin{proof}
For (a), we refer to \cite[Theorem~3.1]{IL}. For (b), we refer to \cite[Proposition~2.12]{DITW}. 

(c) Let $I$ be a divisorial ideal of $R$. Then $(S\otimes_RI)^{**}$ is a divisorial ideal of $S$
which is $G$-graded. Since $\Cl(S)=0$, we have an isomorphism
$(S\otimes_RI)^{**}\simeq S(i)$ of $G$-graded $S$-modules. 
Taking the degree zero part, we have an isomorphism $I\simeq S_i$ of $R$-modules.
\end{proof}

Now we are ready to prove Theorem~\ref{main_thm}.

\begin{proof}[Proof of Theorem~\ref{main_thm}] 
We prove the following implications, where the implications written by the arrows $\xymatrix@C2em{\ar@3{->}[r]&}$ are clear.
\[\xymatrix{
(4)\ar@{=>}[r]\ar@3{->}[d]&(8)\ar@3{->}[r]\ar@3{->}[d]&(6)\ar@3{->}[d]\\
(5)\ar@{=>}[r]&(9)\ar@3{->}[r]&(7)\ar@{=>}[r]&(2)\ar@{=>}[r]&(3)\ar@{=>}[r]&(4)
}\]

(4)$\Rightarrow$(8) (resp. (5)$\Rightarrow$(9)) This follows from Proposition~\ref{cl_finite} (b). 

(7)$\Rightarrow$(2) By Proposition \ref{prepare graded} (b), there exists a strongly $G$-graded ring $S=\bigoplus_{X\in G}X$ such that $S$ is a complete local and $S_0=R$ and $[S_i]=i$ hold in $\Cl(R)$ for any $i\in G$.
We only have to show that $S$ is a regular, or equivalently, the residue field $k$
has a finite projective dimension as an $S$-module.
By Proposition \ref{prepare graded} (a), we have equivalences
\begin{equation*}
\mod^GS\simeq\mod(S\# G)\simeq\mod\End_R(S).
\end{equation*}
Since $S$ gives an NCR of $R$, the abelian category $\mod^GS$ has a finite global dimension.
Thus $k$ has a projective resolution in $\mod^GS$ which has finite length.
Forgetting the grading, the $S$-module $k$ has a finite projective dimension.

 (2)$\Rightarrow$(3). Since $S$ is strongly $G$-graded, we have an isomorphism $S\# G\simeq\End_R(S)$ of rings in \eqref{SG and End}. 
In particular $\End_R(S)$ belongs to $\CM R$. By Proposition \ref{prepare graded}
(a), we have equivalences $\mod^GS\simeq\mod\End_R(S)$.
Since $S$ is regular, the category $\mod^GS$ has finite global dimension at most $\dim S$. Therefore $S$ gives an NCCR of $R$. 

It remains to prove that there are no other splitting $R$-modules giving NCCRs up to additive equivalences. By Proposition \ref{prepare graded} (c), any rank one reflexive $R$-module is a direct summand of $S$. Thus the assertion follows from Proposition \ref{maximal_NCCR}. 

(3)$\Rightarrow$(4) We suppose $N=N_1\oplus\cdots\oplus N_n$ is basic, and it gives the unique splitting NCCR. 
Since $\End_R(\Hom_R(N_i,N))\simeq\End_R(N)$, $\Hom_R(N_i,N)$ also gives a splitting NCCR. 
By the uniqueness, $\Hom_R(N_i,N)\simeq N$ for any $i$. Thus, $\End_R(N)=\bigoplus^n_{i=1}\Hom_R(N_i,N)\simeq N^{\oplus n}$ is a steady NCCR. 

\medskip

Thus we have shown that all the conditions (2)--(9) are equivalent.

In the rest, we assume $R$ contains an algebraically closed field of characteristic zero. 

(1)$\Rightarrow$(4) This is shown in Example~\ref{splitting_example}(a).

(2)$\Rightarrow$(1) 
The dual abelian group $G^\vee$ acts on $S$ by $f(x)=f(i)x$ for any $f\in G^\vee$, $i\in G$ and $x\in S_i$. 
Moreover, the action of $G$ is linearizable (see e.g. \cite[5.3]{LW}), and we can assume that $G$ is a small subgroup of $\GL(d,k)$ (see the proof of Example \ref{steady_example}).
Since $R=S_0$, we have the assertion. 
\end{proof}

Next we prove Corollary \ref{toric}.

\begin{proof}[Proof of Corollary \ref{toric}] 
Suppose that $R\simeq \widehat{A}$ is the $\mm$-adic completion of a toric singularity $A$ where $\mm$ is the irrelevant maximal ideal. 
That is, let $\sfS$ be a positive affine normal semigroup, and let $A:=k[\sfS]$. Here, we note the relationship of class groups. 
We define a natural morphism 
\[
\Cl(A)\rightarrow\Cl(R) \quad \left(\,[I]\mapsto [\,\widehat{I}\simeq I\otimes_A R\,]\,\right). 
\]
This is injective (see e.g. \cite[15.2]{Yos}).

It is clear that the equivalent conditions in Theorem~\ref{main_thm} implies $\Cl(R)$ is a finite group. 
Conversely, we assume $\Cl(R)$ is a finite group. 
Since we have $\Cl(A)\hookrightarrow\Cl(R)$, $\Cl(A)$ is also a finite group. 
Therefore, a semigroup $\sfS$ is a simplicial, and it is equivalent to $A$ is a quotient singularity associated with a finite abelian group (see e.g. \cite[4.59]{BG}, \cite[1.3.20]{CLS}). Consequently, $R$ is also a quotient singularity associated with a finite abelian group.
\end{proof}

Finally we prove Corollary \ref{dimer}.

\begin{proof}[Proof of Corollary \ref{dimer}]
The second and the third conditions are equivalent to each other by Theorem~\ref{main_thm} (1)$\Leftrightarrow$(3)$\Leftrightarrow$(4). It remain to show that they are equivalent to the first condition.

Assume that $\Gamma$ is a regular hexagonal dimer model. 
Then by \cite[Lemma 5.3]{HIO}, there exists a finite abelian subgroup $G$
of $\SL(3,k)$ such that $\mathcal{P}(Q_\Gamma,W_\Gamma)=k[[x_1,x_2,x_3]]*G$.
Thus $R=k[[x_1,x_2,x_3]]^G$ holds.

Assume that $R=S^G$ with $S=k[[x_1,x_2,x_3]]$ and a finite small abelian subgroup $G$ of $\SL(3,k)$. By Theorem~\ref{main_thm} (1)$\Rightarrow$(3), $S$ is a unique $R$-module which gives an NCCR up to additive equivalence.
Thus we have $\mathcal{P}(Q_{\Gamma},W_{\Gamma})\simeq\End_R(S)\simeq S*G$ as $R$-algebras. This implies that $Q_\Gamma$ is the McKay quiver of $G$, and therefore each face of $\Gamma$ is hexagon (see \cite[Section~5]{UY}) and it is homotopy equivalent to a regular hexagonal dimer model.
\end{proof}

\section{Remarks and Questions}

We end this paper by a few remarks. 
We do not know an answer to the following question.

\medskip

\begin{question}
Assume that $R$ contains an algebraically closed field of characteristic zero.
Does existence of steady NCCRs of $R$ imply that $R$ is a quotient singularity?
\end{question}

This is true in dimension two. More strongly, we have the following observation.

\begin{proposition}
\label{twodim}
For a two dimensional complete local normal domain $R$ containing an algebraically closed field of characteristic zero, the following conditions are equivalent.
\begin{enumerate}[\rm(a)]
\setlength{\parskip}{0pt} 
\setlength{\itemsep}{0pt}
\item $R$ is a quotient singularity associated with a finite group $G\subset\GL(2, k)$ (i.e. $R=S^G$ where $S=k[[x_1, x_2]]$ for a field $k$).
\item $R$ has a steady NCCR.
\item $R$ has an NCCR. 
\item $R$ has only finitely many isomorphism classes of indecomposable MCM R-modules.
\end{enumerate}
Moreover, in this case, NCCRs of $R$ are precisely additive generators of $\CM R$, and hence any NCCR of $R$ is steady.
\end{proposition}

\begin{proof}
(a)$\Rightarrow$(b) This is shown in Example \ref{steady_example} (a).

(b)$\Rightarrow$(c) This is clear.

(c)$\Rightarrow$(d) It is well-known that, in dimension two, modules giving NCCRs are precisely additive generators of $\CM R$ (see Proposition~\ref{maximal_NCCR}). 

(d)$\Rightarrow$(a) This is well-known (see \cite[2.1 and 4.9]{Aus3}, \cite[Chapter~11]{Yos}).
\end{proof}

We pose the following question, which we do not know an answer even for quotient singularities. 

\begin{question}
Are all steady NCCRs Morita equivalent?
More strongly, are all modules giving steady NCCRs additive equivalent?
\end{question}

This is true in dimension two by Proposition~\ref{twodim}.

\end{document}